\documentclass[a4paper,11pt]{amsart}

\usepackage{amssymb,amsmath,amsthm,mathrsfs,enumerate,graphicx, color}
\usepackage[pdfpagelabels,colorlinks,linkcolor=blue,citecolor=black,urlcolor=blue]{hyperref}
\usepackage{esint}
\newtheorem{thm}{Theorem}[section]

\newtheorem{cor}[thm]{Corollary}
\newtheorem{lem}[thm]{Lemma}
\newtheorem{prop}[thm]{Proposition}
\newtheorem{defn}[thm]{Definition}
\newtheorem{rem}[thm]{Remark}

\newcommand{\lesi}{\lesssim}

\newcommand{\f}{\frac}

\newcommand{\Om}{\Omega}

\newcommand{\vc}{\infty}
\newcommand{\Rn}{\mathbb{R}^n}

\newcommand{\di}{{\rm div}}

\title[Global Lorentz estimates for parabolic equations]{Global Lorentz estimates for nonlinear parabolic equations on nonsmooth domains}         % Enter your title between curly braces
\author{The Anh Bui}
\address{Department of Mathematics, Macquarie University, NSW 2109,
Australia}
\email{the.bui@mq.edu.au, bt\_anh80@yahoo.com}
 
\author{Xuan Thinh Duong}
\address{Department of Mathematics, Macquarie University, NSW 2109,
Australia}

\email{xuan.duong@mq.edu.au}


\begin{document}

\date{}

\maketitle

\begin{abstract} Consider the nonlinear parabolic equation in the form
$$
u_t-\di \mathbf{a}(D u,x,t)=\di\,(|F|^{p-2}F) \quad \text{in} \quad \Om\times(0,T),
$$	
where $T>0$ and $\Om$ is a Reifenberg domain. We suppose that the nonlinearity $\mathbf{a}(\xi,x,t)$ has a small BMO norm with respect to $x$ and is merely measurable and bounded with respect to the time variable $t$. In this paper, we prove the global Calder\'on-Zygmund estimates for the weak solution to this parabolic problem in the setting of Lorentz spaces which includes the estimates in Lebesgue spaces. Our global Calder\'on-Zygmund estimates extend certain previous results to equations with less regularity assumptions on the nonlinearity $\mathbf{a}(\xi,x,t)$ and to more general setting of Lorentz spaces.
\end{abstract}

\tableofcontents

\section{Introduction}

Let $\Om$ be a bounded open domain in $\Rn$. For $\f{2n}{n+2}<p<\vc$, we consider the following parabolic equation
\begin{equation}\label{ParabolicProblem}
\left\{
\begin{aligned}
&u_t-\di \mathbf{a}(D u,x,t)=\di\,(|F|^{p-2}F) \quad &\text{in}& \quad \Om_T,\\
&u=0 \quad &\text{on}& \quad \partial_p\Om_T,
\end{aligned}\right.
\end{equation}
where $T>0$ is a given positive constant, $\Om_T=\Om\times (0,T)$,  $\partial_p\Om_T=(\partial\Om\times (0,T))\cup (\Om\times\{0\})$, and $F=(F_1,\ldots, F_n)\in L^p(\Om_T, \Rn)$ is a given vector-valued function. Throughout the paper, we denote $u_t=\f{\partial u}{\partial t}$ and $Du=D_xu=(D_{x_1}u,\ldots, D_{x_n}u)$.

In this paper, we assume that the nonlinearity $\mathbf{a}=(\mathbf{a}^1,\ldots,\mathbf{a}^n): \mathbb{R}^n\times \mathbb{R}^n\times \mathbb{R}\to \mathbb{R}^n$ in \eqref{ParabolicProblem} is measurable in $(x,t)$ for every $\xi$ and continuous in $\xi$ for a.e. $(x,t)$, and satisfies the following conditions: there exist $\Lambda_1, \Lambda_2>0$ so that 
	\begin{equation}\label{eq1-functiona}
	|\mathbf{a}(\xi,x,t)|+|\xi||D_\xi\mathbf{a}(\xi,x,t)|\leq \Lambda_1|\xi|^{p-1},
	\end{equation}
	and
	\begin{equation}\label{eq2-functiona}
	\langle \mathbf{a}(\xi,x,t)-\mathbf{a}(\eta,x,t), \xi -\eta\rangle\geq \Lambda_2\begin{cases}
	|\xi-\eta|^p, \ \ & p\geq 2,\\
	(|\xi|+|\eta|)^{p-2}|\xi-\eta|^2, \ \ & \f{2n}{n+2}<p<2.
	\end{cases}
	\end{equation}
	for a.e $\xi,\eta  \in \mathbb{R}^n$ and a.e. $(x,t)\in \mathbb{R}^n\times \mathbb{R}$.

Note that an example of such a nonlinearity $\mathbf{a}(\xi,x,t)$ satisfying these conditions is the
standard $p$-Laplacian $\Delta_p u={\rm div}(|D u|^{p-2}u)=\di a(\xi,x,t)$ corresponding to $\mathbf{a}(\xi,x,t)=|\xi|^{p-2}\xi$.

A function $u\in C(0,T; L^2(\Om))\cap L^p(0, T; W^{1,p}_0(\Om))$ is said to be a weak solution to the problem \eqref{ParabolicProblem} if the following holds true
\begin{equation}
\label{eq-weak solution}
\int_{\Om_T}u\varphi_t dxdt -\int_{\Om_T}\langle {\bf a}(Du, x,t), D\varphi\rangle dxdt = \int_{\Om_T}\langle |F|^{p-2}F, D\varphi\rangle dxdt,
\end{equation} 
for every test function $\varphi \in C^\vc_0(\Om_T)$.

Due to the lack of regularity with respect to the time variable, the weak solution $u$ to the problem \eqref{ParabolicProblem} could not be chosen as a test function in the formula \eqref{eq-weak solution}. In order to overcome this trouble, we make use of the Steklov averages, and take the test function $\varphi=u$ which is possible modulo Steklov averages. For further details on the Stelkov averages and their applications to parabolic equations, we refer to \cite{B}.

It is well-known that if $|F|\in L^p(\Om_T)$, then the equation \eqref{ParabolicProblem} has a unique weak solution $u$ satisfying the following Calder\'on-Zygmund estimate
\begin{equation}
\label{L^p-boundedness}
\sup_{0<t<T}\|u(\cdot, t)\|_{L^2(\Om)}+\|Du\|_{L^p(\Om_T)}\leq C\|F\|_{L^p(\Om_T)}.
\end{equation}
Hence, it is natural to raise the question on extending the  Calder\'on-Zygmund estimate \eqref{L^p-boundedness} to various functions spaces.

The Calder\'on-Zygmund theory for the weak solution to the partial differential equations including both elliptic and parabolic equations has been received a great deal of attention by many mathematicians. See for example \cite{AM, AM1, BF, CP,CZ, F, KL,KL2, KZ, K, Gia, GM, Giu, JK, IS, L, LSU, M, Mo, S, To} and the references therein. The main aim of this paper is to study the regularity problem regarding to the parabolic equation \eqref{ParabolicProblem}. We now list some of works related to the research direction.

\begin{enumerate}[(a)]
	\item In \cite{KL}, the authors proved the regularity (higher integrability) of the weak solutions to the following second order parabolic system in the following general form including the system of $p$-Laplacian type
$$
u_t=\di\, A_i(Du, x,t)+B_i(Du,x,t), i=1,\ldots, N,
$$
where the nonlinearities $A_i$ and $B_i$ satisfy growth conditions and some other conditions. The regularity of very weak solutions to this systems was obtained in \cite{KL2}. We refer to \cite[Section 2]{KL2} for the definition of a very weak solution.

\item In \cite{AM}, the authors considered the parabolic equation in the form
\begin{equation}\label{eq-AM}
u_t-\di\, (a(x,t)|Du|^{p-2}Du)=\di\, (|F|^{p-2}F), p>\f{2n}{n+2},
\end{equation}
in the cylindrical domain $\Om_T=\Om\times (0,T)$, where $\Om$ is a bounded open subset in $\Rn$, and the coefficient $a(x,t)$ belongs to VMO spaces and satisfies $0<\nu\leq a(x,t)\leq L<\vc$. Then they proved the local $W^{1,q}$ regularity for the weak solution to this problem. More precisely, it was shown that if $|F|^p\in L^q_{\rm loc}(\Om_T)$ for $q>1$ then $|Du|^p\in L^q_{\rm loc}(\Om_T)$. It is important to note that they introduced a new technique based on their own results in \cite{AM1}. This technique is an effective tool in studying regularity problems of partial differential equations. Moreover, the Calder\'on--Zygmund estimates with measurable dependence with respect to time was obtained in \cite{DMS}.

\item In \cite{BR1}, the global Calder\'on-Zygmund theory for the weak solution to the problem \eqref{ParabolicProblem} was investigated. It was proved that if $|F|^p\in L^q(\Om_T)$ for $q>1$ then $|Du|^p\in L^q(\Om_T)$. This result extends those in \cite{BW2}. Note that in \cite{BR1}, the assumptions only require the nonlinearity ${\bf a}$ to have a small BMO norms with respect to both $x$ and $t$, and the domain $\Om$ is flat in Reifenberg's sense. 
\end{enumerate}

It is worth noticing that although the estimates for the gradient of the weak solutions to the parabolic problems on Lebesgue spaces $L^p$ have been well-known, the global estimates for the gradient of the weak solutions on the Lorentz spaces for the case $p\neq 2$ are less well-known, and even have not been established so far.

Recently, in \cite{Ba}, the author adapted the technique in \cite{AM, AM1} to extend the result in \cite{AM} to prove the local Lorentz estimates for the gradient of weak solutions to the problem \eqref{eq-AM}. More precisely, assume that the coefficient $a(x,t)$ satisfies a VMO condition with respect to $x$ and satisfies $0<\nu\leq a(x,t)\leq L<\vc$. It was prove that if $|F|\in L^{q,r}(Q_{2R}(z_0))$ for $q>p, 0<r\leq \vc$ and $Q_{2R}(z_0)\subset \Om_T$ then $|Du|\in L^{q,r}(Q_{R}(z_0)$, where $L^{q,r}$ is a Lorentz space. See Definition \ref{defn-Lorentzspaces}.

The main aim of this paper is to prove the global Cader\'on-Zygmund type estimates for the weak solution to the parabolic equation \eqref{ParabolicProblem} on the Lorentz spaces. We now set up the assumptions and then state the main result of the paper.

In what follows, for a measurable function $f$ on a measurable subset $E$ in $\mathbb{R}^{n}$ (or, in $\mathbb{R}^{n}\times (0,\vc)$) we define
$$
\overline{f}_E =\fint_E f =\f{1}{|E|}\int_E f.
$$
We set
$$
\Theta(\mathbf{a},B_r(y))(x,t)=\sup_{\xi\in \mathbb{R}^n\backslash \{0\}}\f{|\mathbf{a}(\xi,x,t)-\overline{\mathbf{a}}_{B_r(y)}(\xi,t)|}{|\xi|^{p-1}}
$$
where
$$
\overline{\mathbf{a}}_{B_r(y)}(\xi,t)=\fint_{B_r(y)}\mathbf{a}(\xi,x,t)dx.
$$
Throughout this paper we always assume that the nonlinearity ${\bf a}$ satisfy \eqref{eq1-functiona} and \eqref{eq2-functiona}. Additionally, we also require the small BMO semi-norm conditions on the nonlinearity ${\bf a}$. 
\begin{defn}
	Let $R_0,\delta>0$. The nonlinearity ${\bf a}$ satisfies the small $(\delta, R_0)$-{\rm BMO} semi-norm condition if 
	\begin{equation}\label{eq2-Assumption}
	[{\bf a}]_{2,R_0}:=\sup_{y\in \mathbb{R}^n}\sup_{0<r\leq R_0, 0<\tau<r^2}\,\fint_{Q_{(r,\tau)}(y)}|\Theta(\mathbf{a},B_r(y))(x,t)|^2dxdt\leq \delta^2.
	\end{equation}
\end{defn}
\begin{rem}
	
	\noindent (a) The nonlinearity ${\bf a}$ as in \eqref{eq2-Assumption} is assumed to be merely measurable only in the time variable $t$ and belong to the class BMO (functions with bounded mean oscillations) as functions of the spatial variables $x$. To see this, we now consider the following example. If  ${\bf a}(\xi,x,t)=b(\xi,x)c(t)$, then \eqref{eq2-Assumption} requires small BMO norm regularity for $b(\xi,\cdot)$, whereas $c(\cdot)$ is just needed to be bounded and measurable. This contrast to those used in \cite{BR1, BW2} in which the nonlinearity ${\bf a}$ is required to belong to the class BMO in both variables $t$ and $x$. Note that the condition \eqref{eq2-Assumption} is similar to that used in \cite{K} to study the parabolic and elliptic equations with VMO coefficients. We refer to \cite{Sa} for the definition of VMO functions.

	\noindent (b) Under the conditions \eqref{eq1-functiona}, \eqref{eq2-functiona} and \eqref{eq2-Assumption}, it is easy to see that for any $\gamma\in [1,\vc)$ there exists $\epsilon>0$ so that 
	$$
	[{\bf a}]_{\gamma,R_0}
	:=\sup_{y\in \mathbb{R}^n}\sup_{0<r\leq R_0, 0<\tau<r^2}\,\fint_{Q_{(r,\tau)}(y)}|\Theta(\mathbf{a},B_r(y))(x,t)|^\gamma dxdt\lesi \delta^\epsilon.
	$$
\end{rem}

Concerning  the underlying domain $\Omega$, we do not assume any smoothness condition on $\Omega$, but the following  flatness condition.
\begin{defn}\label{defn2}
	Let $\delta, R_0>0$. The domain $\Om$ is said to be a $(\delta, R_0)$ Reifenberg flat domain if for every $x\in \partial\Om$ and $0<r\leq R_0$, then there exists a coordinate system depending on $x$ and $r$, whose variables are denoted by $y=(y_1,\dots,y_n)$ such that in this new coordinate system $x$ is the origin and
	\begin{equation}\label{eq1-Assumption}
	B_{r}\cap\{y: y_n>\delta r\}\subset B_{r}\cap \Omega \subset \{y: y_n>-\delta r\}.
	\end{equation}
\end{defn}

\begin{rem}\label{rem1}
	(a) The condition of  $(\delta, R_0)$-Reifenberg flatness condition was first introduced in \cite{R}. This condition does not require any smoothness on the boundary of $\Om$, but
	sufficiently flat in the Reifenberg's sense. The Reifenberg flat domain includes domains with rough boundaries of fractal
	nature, and Lipschitz domains with small Lipschitz constants.
	For further discussions about the Reifenberg domain, we refer to \cite{R, DT, Toro, P} and the references therein. 
	
	(b) If $\Om$ is a $(\delta, R_0)$ Reifenberg domain, then for any $x_0\in \partial \Om$ and $0<\rho<R_0(1-\delta)$ there exists a coordinate systems, whose variables are denoted by $y=(y_1,\ldots, y_n)$ such that in this coordinate systems the origin is some interior point of $\Om$ and $x_0=(0,\ldots, 0, -\f{\delta \rho}{1-\delta})$ and
	$$
	B_{\rho}^+\subset B_{\rho}\cap \Om\subset B_{\rho}\cap  \left\{y: y_n>-\f{2\delta \rho}{1-\delta}\right\}.
	$$
	From now on, we always assume that $0<\delta<1/10$.

	(c) For $x\in \Om$ and $0<r<R_0$, we have
	\begin{equation}
	\label{eq1-Reifenberg domain}
	\f{|B_r(x)|}{|B_r(x)\cap \Om|}\leq \Big(\f{2}{1-\delta}\Big)^n\leq 4^n.
	\end{equation}
\end{rem}
We recall the definition of the Lorentz spaces. 

\begin{defn}
	\label{defn-Lorentzspaces}
Let $0<q<\vc$ and $0<r\leq \vc$, and let $U$ be an open subset in $\Om_T$. The Lorentz space
$L^{q,r}(U)$ is defined as the set of all measurable functions $f$ on $U$ such that
$$
\|f\|_{L^{q,r}(U)}:=
\begin{cases}
\left\{q\int_0^\vc \left[t^q |\{z\in U: |f(z)|>t\}|\right]^{r/q}\f{dt}{t}\right\}^{1/r}<\vc, \ \ &r<\vc,\\
\sup_{t>0}t|\{z\in U: |f(z)|>t\}|^{1/q}, \ \ & r=\vc.
\end{cases}
$$
\end{defn}
In the particular case $q=r$, the weighted Lorentz spaces $L^{q,q}(U)$ coincide with the $L^q(U)$ spaces. We have the inclusion $L^p(U)\subset L^{q,r}(U)$ for all $q>p, 0<r\leq \vc$. 

Our main result is the following theorem.
\begin{thm}\label{mainthm1}
	Let $q\in (p,\vc)$ and $0<r\leq \vc$. Then there
	exists a positive constant $\delta$ such that the following holds. If $|F|\in L^{q,r}(\Om_T)$, the domain $\Omega$ is a $(\delta, R_0)$-Reifenberg flat domain, and the nonlinearity ${\bf a}$  satisfies \eqref{eq1-functiona}, \eqref{eq2-functiona} and \eqref{eq2-Assumption}, then the problem \eqref{ParabolicProblem} has a unique weak  solution $u$ satisfying the estimate:
	\begin{equation}\label{eq-LorentzEs}
	\left\|\,|Du|\,\right\|_{L^{q,r}(\Om_T)} \lesi	\left[\left\|\,|F|\,\right\|_{L^{q,r}(\Om_T)}+1\right]^d,
	\end{equation}
	where 
	$$
	d=\begin{cases}
	\f{p}{2}, \ \ \ & p\geq 2\\
	\f{2p}{2p+np-2n}, \ \  \ & \f{2n}{n+2}<p<2.
	\end{cases},
	$$
\end{thm}
We note that the presence of the exponent $d$ on the right hand side of \eqref{eq-LorentzEs} is reasonable due to the scaling deficit of the problem \eqref{ParabolicProblem}. See for example \cite{AM}.

Some comments in the borderline $q=p$ are in order. In the case $r=p$ this is not difficult as it comes directly from testing
the equation. In the other cases this is not easy and requires to tilt the estimates
around the natural growth exponent.  For instance this was done in \cite{AP} for nonlinear elliptic equations which made use of estimates below the natural growth exponent and Hodge decompositions. In the parabolic case such tools can be replaced by the methods in \cite{BDM} and might be done in future work.

We now point out the technique used in this paper. In most of papers on the global Lorentz estimates for the gradient of the weak solutions to elliptic and parabolic equations, the Hardy-Littlewood maximal function technique is often employed, see for example \cite{BW2,  QHN, MP, MP2}. In particular case of the parabolic equation in the form \eqref{ParabolicProblem} and some of its specific cases, the maximal function technique works well only in the case $p=2$, see for example \cite{BW2, QHN}. However, the case $p\neq 2$ rules out the homogeneity of the parabolic equations. This is a main reason why the maximal function technique may not be applicable in Lorentz settings. To overcome this trouble, we adapt the free maximal function technique in \cite{AM, AM1, Ba, BR1} which make use of the approximation method in \cite{CP} and the Vitali covering lemma.  However, we would like to point out that the {\it priori regularity assumption approach} as in \cite{AM, AM1, BR1} seems not to be suitable to our setting. This requires some improvements and signification modifications in our approach. 

\medskip

We would like to comment on the main contributions of this paper in comparison with some known results.
\begin{enumerate}[{\rm (a)}]
\item The assumption \eqref{eq2-Assumption} we imposed on the nonlinearity ${\bf a}$ is weaker than that in \cite{BR1}. Precisely, the assumption \eqref{eq2-Assumption} on the nonlinearity ${\bf a}$ is assumed to be merely measurable in the time variable $t$ and having a small BMO semi-norm as functions of the spatial variables $x$. Meanwhile, in \cite{BR1} the nonlinearity ${\bf a}$ is required to belong to the class BMO in both variables $t$ and $x$. Moreover, our paper treats the global estimates in Lorentz settings which extends the global $L^p$ estimates in \cite{BR1}. 

\item The parabolic problem of the form \eqref{eq-AM} studied in \cite{AM,Ba} is a particular case of our equation \eqref{ParabolicProblem} corresponding to ${\bf a}(\xi,x,t)=a(x,t)|\xi|^{p-2}\xi$. Moreover, our small BMO semi-norm condition \eqref{eq2-Assumption} with respect to $x$ is weaker than the VMO condition in \cite{Ba} and the regularity condition in \cite{AM}. More importantly, in this paper, we prove the global estimates for the weak solution rather than the local estimates as in \cite{AM,Ba}. This difference requires some challenging comparison estimates near the boundary. See Section 2.2.  

\item The global weighted Lorentz estimates for the gradient of the weak solution to the problem \eqref{ParabolicProblem} corresponding to $p=2$ was obtained in \cite{QHN} by using the maximal function technique. As mentioned earlier, this technique is not applicable to our settings due to the lack of homogeneity, and in the particular case $p=2$, we recover the results in \cite{QHN} for the unweighted Lorentz estimates by a different approach. 
\end{enumerate}

The organization of the paper is as follows. In Section 2, we prove some interior and boundary estimates for the weak solution to the problem \eqref{ParabolicProblem}. The proof of Theorem \ref{mainthm1} will be represented in Section 3.

Throughout the paper, we always use $C$ and $c$ to denote positive constants that are independent of the main parameters involved but whose values may differ from line to line. We will write
$A\lesi B$ if there is a universal constant $C$ so that $A\leq CB$ and $A\sim B$ if $A\lesi B$ and $B\lesi A$.

We will end this section with some notations which will be used in the paper.
\begin{itemize}
	\item $B_r=\{y: |y|<r\}$, $B_r^+=B_r\cap \{y=(y_1,\ldots,y_n): y_n>0\}$.
	\item $B_r(x)=x+B_r$, $B^+_r(x)=x+B^+_r$.
	\item $Q_{r,\tau}=B_r\times (-\tau, \tau)$, $Q_r=Q_{r,r^2}$, $Q_{r,\tau}^+=Q_{r,\tau}\cap\{z=(x',x_n,t): x_n>0\}$, and $Q^+_r=Q_{r,r^2}^+$.
	\item For $z=(x,t)$, $Q_{r,\tau}(z)=z+Q_{r,\tau}$, and $Q^+_{r,\tau}(z)=z+Q^+_{r,\tau}$.
	\item For $z=(x,t)$, $Q_{r,\tau}(z)=z+Q_{r,\tau}$, and $Q^+_{r,\tau}(z)=z+Q^+_{r,\tau}$.
	\item $K_{r,\tau}(z)=Q_{r,\tau}(z)\cap \Om_T, K_{r}=Q_{r}\cap \Om_T$.
	\item $\partial_w K_r=Q_r\cap (\partial \Om\times \mathbb{R})$.
\end{itemize}

\section{Interior and boundary estimates for the weak solutions}
%Recall that a function $u\in C(0,T; L^2(\Om))\cap L^p(0, T; W^{1,p}_0(\Om))$ is said to be a (weak) solution to the problem \eqref{ParabolicProblem} if the following holds true
%\begin{equation*}
%\int_{\Om_T}u\varphi_t dxdt -\int_{\Om_T}\langle {\bf a}(Du, x,t), D\varphi\rangle dxdt = \int_{\Om_T}\langle |F|^{p-2}F, D\varphi\rangle dxdt,
%\end{equation*} 
%for every test function $\varphi \in C^\vc_0(\Om_T)$.

%Since the weak solution $u\in C(0,T; L^2(\Om))\cap L^p(0, T; W^{1,p}_0(\Om))$, we can see that $u$ is the unique weak solution to \eqref{ParabolicProblem} in $\Om\times (-\vc, T)$, where $F\in L^p(\Om_T, \Rn)$, $F=0$ and $u=0$ in $\Om\times (-\vc, 0).$

%Due to some technical reason, in this section we may assume that $u\in C(-\vc,T; L^2(\Om))\cap L^p(-\vc, T; W^{1,p}_0(\Om))$ is a weak solution to weak solution to \eqref{ParabolicProblem} in $\Om\times (-\vc, T)$, where $F\in L^p(\Om_T, \Rn)$, $F=0$ and $u=0$ in $\Om\times (-\vc, 0).$

{\it In this section, we always assume that the nonlinearity ${\bf a}$ satisfies \eqref{eq1-functiona}, \eqref{eq2-functiona} and the small {\rm BMO} semi-norm condition \eqref{eq2-Assumption}, and the underlying domain is a $(\delta, R_0)$ Reifenberg domain with $R_0=6$ and small $\delta\in (0,1/10)$ which will be determined later.}
\subsection{Interior estimates}
For $z_0=(x_0,t_0)\in \Om\times (0,T)$ such that $B_5\equiv B_5(z_0)\subset \Om$. For simplicity we may assume that $x_0=0$ and $Q_5(z_0)\equiv Q_5\subset \Om_T$.

We have the following higher integrability result. See for example \cite{KL}.
\begin{prop}
	\label{higherInte-Prop1-inter}
	Let $u$ be a weak solution to the problem
	\[
	u_t-\di\, \mathbf{a}(D u,x,t)=\di (|H|^{p-2}H) \quad \text{\rm in} \quad Q_5
	\]
	with $H\in L^{p+\bar{\sigma}}(Q_5)$ for some $\bar{\sigma}>0$. Then there exists $\epsilon_0\in (0,\bar{\sigma}]$ and $\sigma>0$ so that $|Du|\in L^{p+\epsilon_0}(Q_4)$ and
	\[
	\fint_{Q_4}|Du|^{p+\epsilon_0}dxdt\lesi \Big(\fint_{Q_5}|Du|^pdxdt\Big)^\sigma + \fint_{Q_5}|H|^{p+\epsilon_0}dxdt+1. 
	\]
\end{prop}
Let $u$ be a weak solution to \eqref{ParabolicProblem}. We now consider the unique weak solution $w \in C(t_0-5^2,t_0+5^2; L^2(\Om))\cap L^p(t_0-5^2,t_0+5^2; W^{1,p}_0(\Om))$ to the following equation
\begin{equation}\label{AppProb1-interior}
\left\{
\begin{aligned}
&w_t-\di\, \mathbf{a}(D w,x,t)=0 \quad &\text{in}& \quad Q_5,\\
&w=u \quad &\text{on}& \quad \partial_p Q_5.
\end{aligned}\right.
\end{equation}

We now prove the following useful result.
\begin{lem}\label{lem1-inter}
	Let $w$ be a weak solution to the problem \eqref{AppProb1-interior}. Then for each $\epsilon>0$ there exists $C>0$ so that
	\begin{equation}
	\label{eq-lem1-inter}
	\fint_{Q_5}|D(u-w)|^pdxdt\leq \epsilon\fint_{Q_5}|Du|^pdxdt+C\fint_{Q_5}|F|^pdxdt.
	\end{equation}
\end{lem}
\begin{proof}
	The proof of this lemma is standard. However, for the sake of completeness, we provide it here.
	
	\noindent{\bf Case 1: $\f{2n}{n+2}<p<2$.} Observe that
	$$
	|D(u-w)|^p= (|Du|+|Dw|)^{-\f{p(p-2)}{2}}(|Du|+|Dw|)^{\f{p(p-2)}{2}}|D(u-w)|^p.
	$$
	Hence, for $\tau>0$, using Young's inequality we obtain
	\begin{equation}\label{eq1-proof prop1 inter}
	\begin{aligned}
	\fint_{Q_5}|D(u-w)|^pdxdt&\leq \tau\fint_{Q_5}(|Du|+|Dw|)^pdxdt + c(\tau)\fint_{Q_5}(|Du|+|Dw|)^{p-2}|D(u-w)|^2dxdt\\
	&\leq c(p)\tau\Big[\fint_{Q_5}|D(u-w)|^pdxdt +\fint_{Q_5}|Du|^pdxdt\Big]\\
	& \ \ + c(\tau)\fint_{Q_5}(|Du|+|Dw|)^{p-2}|D(u-w)|^2dxdt.
	\end{aligned}
	\end{equation}
	
	Note that, by \eqref{eq2-functiona}, we have
	\begin{equation}\label{eq2-proof prop1 inter}
	\begin{aligned}
	\fint_{Q_5}(|Du|+|Dw|)^{p-2}|D(u-w)|^2dxdt \leq C\fint_{Q_5} \langle {\bf a}(Du,x,t)-{\bf a}(Dw,x,t), Du-Dw \rangle dxdt.
	\end{aligned}
	\end{equation}
	Taking $u-w$ as a test function, we can verify that
	$$
	\fint_{Q_5}\langle {\bf a}(Du,x,t)-{\bf a}(Dw,x,t), Du-Dw \rangle dxdt=-\fint_{Q_5}\langle |F|^{p-2}F,Du-Dw \rangle dxdt.
	$$
	This along with Young's inequality again implies $\upsilon$
	\begin{equation}\label{eq3-proof prop1 inter}
	\fint_{Q_5}\langle {\bf a}(Du,x,t)-{\bf a}(Dw,x,t), Du-Dw \rangle dxdt\leq \tau\fint_{Q_5}|D(u-w)|^p dxdt+ C\fint_{Q_5}|F|^p dxdt.
	\end{equation}
	Taking \eqref{eq1-proof prop1 inter}, \eqref{eq2-proof prop1 inter} and \eqref{eq3-proof prop1 inter} into account, we obtain
	$$
	\begin{aligned}
	\fint_{Q_5}|D(u-w)|^pdxdt&\leq (c(p)+1)\tau\fint_{Q_5}|D(u-w)|^pdxdt +c(p)\tau\fint_{Q_5}|Du|^pdxdt+C\fint_{Q_5}|F|^p dxdt.
	\end{aligned}	
	$$ 
	By taking $\tau$ to be sufficiently small, this follows the desired estimate.
	
	\medskip
	
	\noindent {\bf Case 2: $p\geq 2$.} This case can be done in the same manner and we omit details.
\end{proof}

%The standard argument implies that 
%\begin{equation}
%\label{eq-u-w-interior}
%\fint_{Q_5}|w|^pdxdt\leq C\fint_{Q_6}\left[|u|^p+|F|^p\right] dxdt.
%\end{equation}

Let $w$ be a weak solution to \eqref{AppProb1-interior}. We now consider the following problem
\begin{equation}\label{AppProb2-interior}
\left\{
\begin{aligned}
&v_t-\di\, \overline{\mathbf{a}}_{B_4}(D v,t)=0 \quad &\text{in}& \quad Q_4,\\
&v=w \quad &\text{on}& \quad \partial_p Q_4.
\end{aligned}\right.
\end{equation}

We then obtain the following estimate.

\begin{lem}
	\label{lem2-inter}
	Let $v$ solve \eqref{AppProb2-interior}. Then for any $\epsilon>0$ there exists $C>0$ and $\sigma_1$ (which is independent of $\epsilon$) so that
	\begin{equation}
	\label{eq-lem2-inter}
	\fint_{Q_4}|D(w-v)|^pdxdt\leq \epsilon \fint_{Q_4}|Dw|^pdxdt +C[{\bf a}]_{2,R_0}^{\sigma_1}\Big(\fint_{Q_5}|Dw|^{p}  dxdt\Big)^{\f{\sigma(p+\epsilon_0)}{p}}.
	\end{equation}
\end{lem}
\begin{proof}
	Arguing similarly to \eqref{eq1-proof prop1 inter} and \eqref{eq2-proof prop1 inter}, for $\tau>0$, we obtain that
	$$
	\fint_{Q_4}|D(w-v)|^pdxdt\leq \tau_1 \fint_{Q_4}|Dw|^pdxdt+C\fint_{Q_4} \langle \overline{{\bf a}}_{B_4}(Dw,t)-\overline{{\bf a}}_{B_4}(Dv,t), Dw-Dv \rangle dxdt.
	$$
	Taking $w-v$ as a test function, it can be verified that
	$$
	\begin{aligned}
	\fint_{Q_4} \langle \overline{{\bf a}}_{B_4}(Dw,t)-\overline{{\bf a}}_{B_4}(Dv,t), Dw-Dv \rangle dxdt&=\fint_{Q_4} \langle \overline{{\bf a}}_{B_4}(Dw,t)-{{\bf a}}_{B_4}(Dw,x,t), Dw-Dv \rangle dxdt.
	\end{aligned}
	$$
	These two estimates imply
	\begin{equation}\label{eq1-proof w-v}
	\begin{aligned}
	\fint_{Q_4}|D(w-v)|^pdxdt&\leq \tau_1 \fint_{Q_4}|Dw|^pdxdt+C\fint_{Q_4} \langle \overline{{\bf a}}_{B_4}(Dw,t)-{{\bf a}}_{B_4}(Dw,x,t), Dw-Dv \rangle dxdt\\
	&\leq \tau_1 \fint_{Q_4}|Dw|^pdxdt+C\fint_{Q_4} \Theta(\mathbf{a},B_4) |Dw|^{p-1} |D(w-v)|  dxdt.
	\end{aligned}
	\end{equation}
	Applying Young's inequality and Proposition \ref{higherInte-Prop1-inter}, we have
	\begin{equation}\label{eq2-proof w-v}
	\begin{aligned}
	\fint_{Q_4} &\Theta(\mathbf{a},B_4) |Dw|^{p-1} |D(w-v)|  dxdt\\
	&\leq \tau_2\fint_{Q_4}|D(w-v)|^p + C\fint_{Q_4} \Theta(\mathbf{a},B_4)^{\f{p}{p-1}} |Dw|^p  dxdt\\
	&\leq \tau_2\fint_{Q_4}|D(w-v)|^p + C\Big(\fint_{Q_4} \Theta(\mathbf{a},B_4)^{\f{p(p+\epsilon_0)}{(p-1)\epsilon_0}}dxdt\Big)^{\f{\epsilon_0}{p+\epsilon_0}} \Big(\fint_{Q_4}|Dw|^{p+\epsilon_0}  dxdt\Big)^{\f{p}{p+\epsilon_0}}\\
	&\leq \tau_2\fint_{Q_4}|D(w-v)|^p +[{\bf a}]_{2,R_0}^{\sigma_1}\Big(\fint_{Q_5}|Dw|^{p}  dxdt\Big)^{\f{\sigma(p+\epsilon_0)}{p}}.
	\end{aligned}
	\end{equation}
	From \eqref{eq1-proof w-v} and \eqref{eq2-proof w-v}, by taking $\tau_1$ and $\tau_2$ to be sufficiently small, we obtain the desired estimate.
\end{proof}
%The standard $L^p$ estimates again yields that
%\begin{equation}
%\label{eq-w-v-interior}
%\fint_{Q_4}|v|^pdxdt\leq C\fint_{Q_5}|w|^p dxdt.
%\end{equation}

We now state the standard H\"older regularity result. See for example \cite[Chapter 8]{B}.
\begin{prop}\label{Lvc-Prop-inter}
	Let $v$ solve the equation \eqref{AppProb2-interior}. Then we have
	$$
	\|Dv\|_{L^\vc(Q_3)}\leq C\Big(\fint_{Q_4}|Dv|^pdxdt+1\Big)^{1/p}.
	$$
\end{prop}

We have the following approximation result.

\begin{prop}
	\label{appProp3-inter}
	For each $\epsilon>0$ there exists $\delta>0$ so that the following holds true. Assume that $u$ is a weak solution to the problem \eqref{ParabolicProblem} satisfying
	\begin{equation}
	\label{eq-u-prop3}
	\fint_{Q_5}|Du|^pdxdt\leq 1,
	\end{equation}
	under the condition 
	\begin{equation}
	\label{eq-F-prop3}
	\fint_{Q_5}|F|^pdxdt\leq \delta^p.
	\end{equation}
	Then there exists a weak solution $v$ to the problem \eqref{AppProb2-interior} satisfying
	\begin{equation}
	\label{eq-Lvc v-prop3}
	\|Dv\|_{L^\vc(Q_3)}\lesi 1,
	\end{equation}
	and
	\begin{equation}
	\label{eq-u-v-prop3}
	\fint_{Q_3}|D(u-v)|^pdxdt\leq \epsilon^p.
	\end{equation}
\end{prop} 
\begin{proof}
	The inequality \eqref{eq-u-v-prop3} follows immediately from Lemma \ref{lem1-inter} and \ref{lem2-inter} and the following estimate
	$$
 	 \fint_{Q_3}|D(u-v)|^pdxdt\lesi \fint_{Q_3}|D(u-w)|^pdxdt+\fint_{Q_3}|D(w-v)|^pdxdt.
	$$
	From Proposition \ref{Lvc-Prop-inter}, we have
	$$
	\|Dv\|^p_{L^\vc(Q_3)}\leq C\fint_{Q_4}|Dv|^p dxdt +C\leq C\fint_{Q_4}|Du|^p dxdt+C\fint_{Q_4}|D(u-v)|^p dxdt +C.
	$$
	This along with \eqref{eq-u-prop3} and \eqref{eq-u-v-prop3} yields \eqref{eq-Lvc v-prop3}.
	
\end{proof}
 \subsection{Boundary estimates}
We now consider the boundary case. Fix $t_0\in (0, T)$ and $z_0=(x_0,t_0)\in \Om_T$. We may assume that $x_0=0$. Without loss of generality we may assume that
\begin{equation}
\label{geometriccondition}
B_5^+\subset \Om_5\subset \Om_5\cap \{x: x_n>-12\delta\}.
\end{equation}
Without loss of geneality we may assume that $(t_0-5^2,t_0+5^2)\subset (0,T)$.

Similarly to Proposition  \ref{higherInte-Prop1-inter}, the higher integrability result still holds true near the boudary of the domain $\Om$. See for example \cite{P, P2, BP}.
\begin{prop}
	\label{higherInte-Prop1-boundary}
Let $u$ be a weak solution to the problem
\[
\left\{
\begin{aligned}
&u_t-\di\, \mathbf{a}(D u,x,t)=\di(|H|^{p-2}H) \quad &\text{\rm in}& \quad K_5(z_0),\\
&u=0 \quad &\text{on}& \quad \partial_wK_5(z_0).
\end{aligned}\right.
\]
with $H\in L^{p+\bar{\sigma}}(K_5(z_0))$ for some $\bar{\sigma}>0$. Then there exists $\epsilon_0\in (0,\bar{\sigma}]$ and $\sigma>0$ so that $|Du|\in L^{p+\epsilon_0}(K_4(z_0))$ and
\[
\fint_{K_4(z_0)}|Du|^{p+\epsilon_0}dxdt\lesi \Big(\fint_{K_5(z_0)}|Du|^pdxdt\Big)^\sigma + \fint_{K_5(z_0)}|H|^{p+\epsilon_0}dxdt+1. 
\]	
\end{prop}

Let $u$ be a weak solution to the problem \eqref{ParabolicProblem}. We consider the  unique weak solution 
$$
w\in C(t_0-5^2,t_0+5^2; L^2(\Om\cap B_5))\cap L^p(t_0-5^2,t_0+5^2; W^{1,p}_0(\Om\cap B_5))
$$
to the  following equation
\begin{equation}\label{AppProb1-boundary}
\left\{
\begin{aligned}
&w_t-\di\, \mathbf{a}(D w,x,t)=0 \quad &\text{in}& \quad K_5(z_0),\\
&w=u \quad &\text{on}& \quad \partial_pK_5(z_0).
\end{aligned}\right.
\end{equation}

Similarly to Lemma \ref{lem1-inter}, we can prove the following result.
\begin{lem}\label{lem1-boundary}
	Let $w$ be a weak solution to the problem \eqref{AppProb1-boundary}. Then for each $\epsilon>0$ there exists $C>0$ so that
	\begin{equation}
	\label{eq-lem1-boundary}
	\fint_{K_5(z_0)}|D(u-w)|^pdxdt\leq \epsilon\fint_{K_5(z_0)}|Du|^pdxdt+C\fint_{K_5(z_0)}|F|^pdxdt.
	\end{equation}
\end{lem}

Let $w$ be a weak solution to \eqref{AppProb1-boundary}. We now consider the following problem
\begin{equation}\label{AppProb2-boundary}
\left\{
\begin{aligned}
&h_t-\di\, \overline{\mathbf{a}}_{B_4}(D h,t)=0 \quad &\text{in}& \quad K_4(z_0),\\
&h=w \quad &\text{on}& \quad \partial_p K_4(z_0).
\end{aligned}\right.
\end{equation}
Using the argument as in the proof of Lemma \ref{lem2-inter} we obtain the following estimate.
\begin{lem}
	\label{lem2-boundary}
	Let $h$ solve \eqref{AppProb2-boundary}. Then for any $\epsilon>0$ there exists $C>0$ and $\sigma_1$ (which is independent to $\epsilon$) so that
	\begin{equation}
	\label{eq-lem2-boundary}
	\fint_{K_4(z_0)}|D(w-h)|^pdxdt\leq \epsilon \fint_{K_4(z_0)}|Dw|^pdxdt +C[{\bf a}]_{2,R_0}^{\sigma_1}\Big(\fint_{K_5(z_0)}|Dw|^{p}  dxdt\Big)^{\f{\sigma(p+\epsilon_0)}{p}}.
	\end{equation}
\end{lem}
However, the main trouble is that the $L^\vc$-norm of the weak solution $h$ may not be bounded near the boundary due to the lack of the smoothness of the domain $\Om$. To overcome this trouble, we now consider its reference problem
\begin{equation}\label{AppProb3-boundary}
\left\{
\begin{aligned}
&v_t-\di\, \overline{\mathbf{a}}_{B_4}(D v,t)=0 \quad &\text{in}& \quad Q_4^+(z_0),\\
&v=0 \quad &\text{on}& \quad Q_4(z_0)\cap \{z=(x',x_n,t): x_n=0\}.
\end{aligned}\right.
\end{equation}
\begin{defn}
	A weak solution $v$ to the problems \eqref{AppProb3-boundary} is understood in the following sense: the zero extension $\bar{v}$ of $v$ is in  $ C(t_0-4^2, t_0+4^2; L^2(B_4))\cap L^p(t_0-4^2, t_0+4^2; W^{1,p}_0(B_4))$ and satisfies the following
	\begin{equation}
	\label{eq-weak solution}
	\int_{Q^+_4(z_0)}h\varphi_t dxdt -\int_{Q^+_4(z_0)}\langle {\bf a}(Dh, x,t), D\varphi\rangle dxdt = 0,
	\end{equation} 
	for every test function $\varphi \in C^\vc_0(Q^+_4(z_0))$.
	\end{defn}

 \begin{prop}
 	\label{approxPropBoundary-1}
 	For every $\epsilon>0$, there exists $\delta$ such that the following holds. If $h$ is a  weak solution to the problem \eqref{AppProb2-boundary} along with \eqref{geometriccondition} and 
 	\begin{equation}
 	\label{eq h-ApproxPro1}
 	\fint_{K_4(z_0)}|D h|^p\lesi 1,
 	\end{equation}
 	then there exists $v$ solving the problem \eqref{AppProb3-boundary} with 
 	\begin{equation}
 	\label{eq1 v-ApproxPro1}
 	\fint_{Q_4^+(z_0)}|Dv|^p\lesi 1
 	\end{equation} 	
 	such that
 	\begin{equation}
 	\label{eq2 v-ApproxPro1}
 	\fint_{Q_4^+(z_0)}|h-v|^p\leq \epsilon^p.
 	\end{equation}
 \end{prop}
 \begin{proof}
 We first note that if $h$ is a weak solution to \eqref{AppProb2-boundary}, then it also solves
 \begin{equation}\label{AppProb2s-boundary}
 \left\{
 \begin{aligned}
 &h_t-\di\, \overline{\mathbf{a}}_{B_4}(D h,t)=0 \quad &\text{in}& \quad K_4(z_0),\\
 &h=0 \quad &\text{on}& \quad \partial_w K_4(z_0).
 \end{aligned}\right.
 \end{equation}
 
 We will argue by contradiction as in \cite{BW1, BR1}. Assume, to the contrary, that there exist an $\epsilon>0$, a sequence of domains $\{\Omega_k\}$ such that
 \begin{equation}\label{geometricconditionOmegak}
 B_5^+\subset \Omega^k_5\subset \{x\in B_5: x_n>-\f{12}{k}\},
 \end{equation}
 and a  sequence of functions $\{h^k\}$ which solves the problem 
 \begin{equation}\label{AppProb2sk-boundary}
 \left\{
 \begin{aligned}
 &h^k_t-\di\, \overline{\mathbf{a}}_{B_4}(D h^k,t)=0 \quad &\text{in}& \quad K^k_{4}(z_0):=(\Om^k\cap B_4)\times (t_0-4^2, t_0+4^2)\\
 &h^k=0 \quad &\text{on}& \quad \partial_w K^k_{4}(z_0).
 \end{aligned}\right.
 \end{equation}
satisfying
 \begin{equation}
 \label{eq h-ApproxPro1 Ok}
 \fint_{K^k_4(z_0)}|D h^k|^p\lesi 1.
 \end{equation}
 But, we have
 \begin{equation}
 \label{eq2 v-ApproxPro1 Ok}
 \fint_{Q_4^+(z_0)}|h^k-v|^p> \epsilon,
 \end{equation}
 where $v$ is any weak solution to the problem \eqref{AppProb3-boundary} with
 \begin{equation}
 \label{eq1 vk-ApproxPro1}
 \fint_{B_4^+}|Dv|^p\lesi 1.
 \end{equation} 	
 From \eqref{geometricconditionOmegak}, \eqref{eq h-ApproxPro1 Ok} and Poincar\'e inequality, we have
 $$
 \fint_{Q_4^+(z_0)}|Dh^k|^pdxdt \leq \fint_{K^k_4(z_0)}|Dh^k|^pdxdt\leq \fint_{K^k_4(z_0)}|Dh^k|^pdxdt\lesi 1, 
 $$
 and
 $$
 \begin{aligned}
 \|h^k_t\|_{L^{p'}(t_0-4^2, t_0+4^2; W^{-1,p'}(B_4^+))}&=\|\di\, \overline{\mathbf{a}}_{B_4}(D h^k,t)\|_{L^{p'}(t_0-4^2, t_0+4^2; W^{-1,p'}(B_4^+))}\\
 	&\leq \|\overline{\mathbf{a}}_{B_4}(D h^k,t)\|_{L^{p'}(t_0-4^2, t_0+4^2; L^{p'}(B_4^+))}\\
 	&\leq \|(Dh^k)^{p-1}\|_{L^{p'}(t_0-4^2, t_0+4^2; L^{p'}(B_4^+))}\\
 	&\lesi  \Big(\int_{K^k_4(z_0)}|D h^k|^p\Big)^{\f{p-1}{p}}\lesi 1.
 \end{aligned}		
 $$
 Therefore, by Aubin-Lions Lemma in \cite[Chapter 3]{Sh}, there exists $h^0$ with $h^0\in L^p(t_0-4^2,t_0+4^2; W^{1,p}(B_4^+))$ and  $h^0_t\in L^{p'}(t_0-4^2,t_0+4^2; W^{-1,p'}(B_4^+))$ such that there exists a subsequence of $\{h^k\}$, which is still denoted by $\{h^k\}$, satisfying
 $$
 h^k\to h^0, \ \ \text{strongly in $L^p(t_0-4^2,t_0+4^2; L^p(B_4^+))$},
 $$
 $$
 Dh^k\to Dh^0, \ \ \text{weakly in $L^p(t_0-4^2,t_0+4^2; L^p(B_4^+))$},
 $$
and
 $$
 h_t^k\to h_t^0, \ \ \text{weakly in  $L^{p'}(t_0-4^2,t_0+4^2; W^{-1,p'}(B_4^+))$}.
 $$
 As a direct consequence, we have
 $$
 \int_{Q_4^+(z_0)}|Dh^0|^pdxdt\lesi \liminf_{k} \int_{Q_4^+(z_0)}|Dh^k|^pdxdt\lesi 1.
 $$
 At this stage, using the method of Browder-Minty as in \cite{BW2}, we can verify that $h^0$ solves
 $$
 \left\{
 \begin{aligned}
 &h^0_t-\di \,\overline{\mathbf{a}}_{B_4}(D h^0,t)=0 \quad &\text{in}& \quad Q_4^+(z_0),\\
 &h^0=0 \quad &\text{on}& \quad Q_4\cap \{x: x_n=0\}\times (t_0-4^2, t_0+4^2).
 \end{aligned}\right.
 $$
  This contradicts to \eqref{eq2 v-ApproxPro1 Ok} by taking $v=h_0$ and $k$ sufficiently large.
\end{proof}	

\begin{prop}
	\label{approxPropBoundary-2}
	For every $\epsilon>0$, there exists $\delta$ such that the following holds. If $h$ is a  weak solution to the problem \eqref{AppProb2s-boundary} along with \eqref{geometriccondition} and 
	\begin{equation}
	\label{eq h-ApproxPro2}
	\fint_{K_4(z_0)}|D h|^p\lesi 1,
	\end{equation}
	then there exists $v$ solving the problem \eqref{AppProb2-boundary} with 
	\begin{equation}
	\label{eq1 v-ApproxPro2}
	\|Dv\|_{L^\vc(Q_3^+(z_0))}^p\lesi 1
	\end{equation} 	
	such that
	\begin{equation}
	\label{eq2 v-ApproxPro2}
	\fint_{K_3(z_0)}|D(h-\bar{v})|^p\leq \epsilon^p,
	\end{equation}
	where $\bar{v}$ is a zero extension of $v$ to $Q_4$.
\end{prop}
\begin{proof}
	Let $\bar{v}$ be a zero extension of $v$ to $Q_4(z_0)$. Then it can be verified that $\bar{v}$ solves 
	$$
	\bar{v}_t-\di\, \overline{\mathbf{a}}_{B_4}(D \bar{v},t)=D_{x_n}\left[\overline{\mathbf{a}}^n_{B_4}(D\bar{v}(x',0,t))\chi_{\{x:x_n<0\}}\right] \ \ \text{in $Q_4(z_0)$},
	$$
	where $x=(x',x_n)$ and $\mathbf{a}=(\mathbf{a}^1,\ldots,\mathbf{a}^n)$.
	
	Therefore, $h-\bar{v}$ solve
	$$
	(h-\bar{v})_t-\di\, \overline{\mathbf{a}}_{B_4}(D (h-\bar{v}),t)= -D_{x_n}\left[\overline{\mathbf{a}}^n_{B_4}(D\bar{v}(x',0,t))\chi_{\{x:x_n<0\}}\right]\quad \text{in} \quad K_4(z_0).
	$$
	By a standard argument and \eqref{eq1-functiona}, we can show that
	$$
	\begin{aligned}
	\fint_{K_3(z_0)}&|D(h-\bar{v})|^pdxdt\\
	&\leq C\fint_{K_4(z_0)}|h-\bar{v}|^p dxdt+C\fint_{K_4(z_0)}|h-\bar{v}|^2 dxdt+C\fint_{K_4(z_0)}|D\bar{v}(x',0,t)\chi_{\{x:x_n<0\}}|^pdxdt.
	\end{aligned}
	$$ 
	Using a similar argument in \cite[pp.4304-4305]{BR1} we obtain that 
	\[
	\fint_{K_4(z_0)}|h-\bar{v}|^2 dxdt\lesi \mathcal{O}(\epsilon).
	\]
	Using \eqref{geometriccondition} and \eqref{eq2 v-ApproxPro1}, we discover that 
	$$
	\fint_{K_4(z_0)}|h-\bar{v}|^pdxdt\leq C\fint_{Q^+_4(z_0)}|h-\bar{v}|^pdxdt+\fint_{K_4(z_0)\backslash Q^+_4(z_0)}|h|^pdxdt\leq C(\epsilon_1 +O(\delta)).
	$$
	By \eqref{geometriccondition}, we have
	$$
	\begin{aligned}
	\fint_{K_4(z_0)}|D\bar{v}(x',0,t)\chi_{\{x:x_n<0\}}|^pdxdt&\leq \fint_{K_4(z_0)\cap \{x: -12\delta<x_n\leq 0\}\times (t_0-4^2,t_0)}|D\bar{v}(x',0,t)|^pdxdt\\
	&\leq O(\delta).
	\end{aligned}
	$$
	These three estimates imply \eqref{eq2 v-ApproxPro2}.
	
	The assertion \eqref{eq1 v-ApproxPro2} follows immediately from \eqref{eq1 v-ApproxPro1} and the H\"older estimate of $v$ near  the flat boundary in \cite{L}:
	$$
	\|v\|^p_{L^\vc(Q_3^+(z_0))}\leq C\fint_{Q_4^+(z_0)}|Dv|^p\leq C.
	$$
	This completes our proof.
\end{proof}
    	
From estimates above, we have the following corollary.
\begin{cor}\label{cor1}
For each $\epsilon>0$ there exists $\delta>0$ so that the following holds true. Assume that $u$ is a weak solution to the problem \eqref{ParabolicProblem} satisfying
\begin{equation}
\label{eq-u-cor}
\fint_{K_5(z_0)}|Du|^pdxdt\leq 1,
\end{equation}
under the condition 
\begin{equation}
\label{eq-F-cor}
\fint_{K_5(z_0)}|F|^pdxdt\leq \delta^p.
\end{equation}
Then there exists a weak weak solution $v$ to the problem \eqref{AppProb2-boundary} with 
\begin{equation}
\label{eq1 v-cor}
\|Dv\|_{L^\vc(Q_3^+(z_0))}^p\lesi 1
\end{equation} 	
such that
\begin{equation}
\label{eq2 v-cor}
\fint_{K_3(z_0)}|D(u-\bar{v})|^p\leq \epsilon^p,
\end{equation}
where $\bar{v}$ is a zero extension of $v$ to $Q_4$.
\end{cor}
\begin{proof}
	Let $w$ and $h$ be, respectively, weak solutions to the problems \eqref{AppProb1-boundary} and \eqref{AppProb2-boundary}. From \eqref{eq-lem1-boundary}, \eqref{eq-lem2-boundary}, \eqref{eq-u-cor} and \eqref{eq-F-cor}, we obtain that 
	$$
	\fint_{K_4(z_0)}|D h|^p\lesi 1.
	$$
	Hence, by Proposition \eqref{approxPropBoundary-2}, we get \eqref{eq1 v-cor}. The estimate \eqref{eq2 v-cor} follows immediately from \eqref{eq-lem1-boundary}, \eqref{eq-lem2-boundary}, \eqref{eq2 v-ApproxPro2} 
	and the following estimate
	$$
	\fint_{K_3(z_0)}|D(u-\bar{v})|^p\lesi \fint_{K_3(z_0)}|D(u-w)|^p+\fint_{K_3}|D(w-h)|^p+\fint_{K_3(z_0)}|D(h-\bar{v})|^p.
	$$    
\end{proof}
    	
\section{The main result}

This section is devoted to prove Theorem \ref{mainthm1}.

Let $F\in L^{q,r}(\Om_T)$ with $q>p, 0<r\leq \vc$ and let $u$ be a unique weak solution to the equation \eqref{ParabolicProblem}.

Fix $1\leq s_1<s_2\leq 2$ and $R<\min\{R_0, 1\}$. Fix $z_0\in \Om_T$.
 
Without loss of generality, we may assume that $\eta=p+\epsilon_0<q$ where $\epsilon_0$ is a constant in Propositions \ref{higherInte-Prop1-inter} and \eqref{higherInte-Prop1-boundary}. Since $F\in L^{q,r}(\Om_T)\subset  L^\eta(\Om_T)\subset L^p(\Om_T)$, from \eqref{L^p-boundedness} we have $|Du|\in L^p(\Om_T)$. We set
\begin{equation*}
\lambda_0:=\Big(\fint_{K_{2R}(z_0)}|Du|^{p}dxdt\Big)^{\f{d}{p}}+\Big(\f{1}{\delta^{\eta}}\fint_{K_{2R}(z_0)}(|F|^{\eta}+1)dxdt\Big)^{\f{d}{\eta}}<\vc.
\end{equation*}
For $\lambda>0$, we now define the level set
$$
E_{s_1}(\lambda)=\{z\in K_{s_1R}(z_0): |Du(z)|>\lambda\}.
$$
For $r>0$, $\lambda>1$ and $z\in K_{s_1R}(z_0)$ we define 
$$
Q_r^\lambda(z)=\begin{cases}
Q_{r,\lambda^{2-p}r^2}(z),  \ \ \ \ &p\geq 2\\
Q_{\lambda^{\f{p-2}{2}}r,r^2}(z),  \ \ \ \ & \f{2n}{n+2}<p< 2,
\end{cases}
$$
and $K_{r}^\lambda(z)=Q_{r}^\lambda(z)\cap \Om_T$.

For $z\in E_{s_1}(\lambda)$, we now define $$
G_z(r)=\Big(\fint_{K_{r}^\lambda(z)}|Du|^{p}dxdt\Big)^{\f{1}{p}}+\Big(\f{1}{\delta^{\eta}}\fint_{K_{r}^\lambda(z)}|F|^{\eta}dxdt\Big)^{\f{1}{\eta}}.
$$
By Lebesgue's differential theorem, we have
\begin{equation}\label{eq Gz0}
\lim_{r\to 0}G_z(r)=|Du(z)|+\f{1}{\delta}|F(z)|>\lambda.
\end{equation}
Then for $\f{(s_2-s_1)R}{100}<r\leq (s_2-s_1)R$ and $\lambda>1$ we have
\begin{equation*}
\begin{aligned}
G_z(r)&=\Big(\fint_{K_{r}^\lambda(z)}|Du|^{p}dxdt\Big)^{\f{1}{p}}+\Big(\f{1}{\delta^{\eta}}\fint_{K_{r}^\lambda(z)}|F|^{\eta}dxdt\Big)^{\f{1}{\eta}}\\
&\leq \left[\f{|K_{2R}(z_0)|}{|K_{r}^\lambda(z)|}\right]^{\f{1}{p}}\Big(\fint_{K_{2R}(z_0)}|Du|^{p}dxdt\Big)^{\f{1}{p}}+\left[\f{|K_{2R}(z_0)|}{|K_{r}^\lambda(z)|}\right]^{\f{1}{\eta}}\Big(\f{1}{\delta^{\eta}}\fint_{K_{2R}(z_0)}|F|^{\eta}dxdt\Big)^{\f{1}{\eta}}\\	
&\leq \left[\f{|K_{2R}(z_0)|}{|K_{r}^\lambda(z)|}\right]^{\f{1}{p}}\lambda_0^{\f{1}{d}}\\
&\leq \left[\f{(2R)^{n+2}}{|K_{r}^\lambda(z)|}\right]^{\f{1}{p}}\lambda_0^{\f{1}{d}}\\
\end{aligned}
\end{equation*}
We note that for $0<r<(s_2-s_1)R$ and $z\in E_{s_1}(\lambda)$, we have $K_r^\lambda(z)\subset K_{2R}(z_0)$.

If $p\geq 2$,  from \eqref{eq1-Reifenberg domain} we have
\begin{equation}
\label{eq1-Gz}
\begin{aligned}
(G_z(r))^p&\leq \f{4^n (2R)^{n+2}}{r^{n+2}\lambda^{{2-p}}} \lambda_0^{\f{p}{d}}\leq 4^n\Big(\f{2R}{r}\Big)^{n+2}\lambda^{p-2}\lambda_0^{\f{p}{d}}\leq 4^n\Big(\f{2\times 10^6}{s_2-s_1}\Big)^{n+2}\lambda^{p-2}\lambda_0^{\f{p}{d}}.
\end{aligned}
\end{equation}

If $\f{2n}{n+2}<p<2$,  similarly we have
\begin{equation}
\label{eq2-Gz}
\begin{aligned}
(G_z(r))^p&\leq 4^n\Big(\f{2\times 10^6}{s_2-s_1}\Big)^{n+2}\lambda^{-\f{(p-2)n}{2}}\lambda_0^{\f{p}{d}}.
\end{aligned}
\end{equation}
We now fix 
$$
\lambda>\Big[4^n\Big(\f{2\times 10^6}{s_2-s_1}\Big)^{n+2}\Big]^{d/p}\lambda_0=\tilde{C}_0 \lambda_0.
$$
Then from \eqref{eq1-Gz} and \eqref{eq2-Gz}, by a simple calculation we obtain
$$
G_z(r)<\lambda, \ \ \ \text{for all $r\in [10^{-6}(s_2-s_1)R,(s_2-s_1)R]$}.
$$
This together with \eqref{eq Gz0} implies that for each $z\in E(\lambda, K_R)$ there exists $0<r_z<10^{-6}(s_2-s_1)R$ so that 
$$
G_z(r_z)=\lambda, \ \ \ \text{and $G_z(r)<\lambda$ for all $r\in (r_z, (s_2-s_1)R)$}.
$$
We now apply Vitali's covering lemma to obtain the following result directly.

\begin{lem}
	\label{coveringlemma}
	There exists a disjoint family $\{K_{r_i}^\lambda(z_i)\}_{i=1}^\vc$ with  $r_i<10^{-6}(s_2-s_1)R$ and $z_i=(x_i,t_i)\in \Om_T$ such that:
\begin{enumerate}[{\rm (a)}]
	\item $E_{s_1}(\lambda)\subset \bigcup_{i}K_{5r_i}^\lambda(z_i)$;
	\item $G_{z_i}(r_i)=\lambda$,  and $G_{z_i}(r)<\lambda$ for all $r\in (r_i, (s_2-s_1)R)$.
\end{enumerate}
\end{lem}

\begin{prop}
	\label{prop1}
	For each $i$ we have
\begin{equation}\label{eq1-wKilambda}
	\begin{aligned}
	|K_{r_i}^\lambda(z_i)|
	&\lesi |K_{r_i}^\lambda(z_i)\cap E_{s_2}(\lambda/4)|+\f{c}{(\delta\lambda)^\eta}\int_{\delta \lambda}^\vc t^{\eta}|\{z\in K_{r_i}^\lambda(z_i): |F(z)|>t\}|\f{dt}{t}.
	\end{aligned}
	\end{equation}
\end{prop}

\begin{proof}
	From Lemma \ref{coveringlemma} we have either
	\[
	\fint_{K_{r_i}^\lambda(z_i)}|Du|^{p}dxdt\geq \f{\lambda^p}{2^p}, \ \ \text{or}  \ \ \ \ \f{1}{\delta^{\eta}}\fint_{K_{r_i}^\lambda(z_i)}|F|^{\eta}dxdt\geq \f{\lambda^\eta}{2^\eta}
	\]
	
	\noindent\textbf{Case 1.} If 
	$$
	\f{1}{\delta^{\eta}}\fint_{K_{r_i}^\lambda(z_i)}|F|^{\eta}dxdt\geq \f{\lambda^\eta}{2^\eta},
	$$
	then we have
	$$
	\begin{aligned}
	|K_{r_i}^\lambda(z_i)|&\leq \f{\eta 2^\eta}{\delta^\eta\lambda^\eta}\int_0^\vc t^{\eta}|\{z\in K_{r_i}^\lambda(z_i): |F(z)|>t\}|\f{dt}{t}\\
	&\leq \int_{0}^{\delta \lambda/4}\ldots +\int_{\delta \lambda/4}^\vc\ldots \\
	&\leq \f{|K_{r_i}^\lambda(z_i)|}{\delta^\eta 2^\eta}+\f{\eta 2^\eta}{\delta^\eta\lambda^\eta}\int_{\delta \lambda/4}^\vc t^{\eta}|\{z\in K_{r_i}^\lambda(z_i): |F(z)|>t\}|\f{dt}{t}.
	\end{aligned}
	$$
	This implies \eqref{eq1-wKilambda}.

	\noindent\textbf{Case 2.} If 
	\[
	\fint_{K_{r_i}^\lambda(z_i)}|Du|^{p}dxdt\geq \f{\lambda^p}{2^p},
	\]
	then
	$$
	|K_{r_i}^\lambda(z_i)|\leq \f{2^p}{\lambda^p} \int_{K_{r_i}^\lambda(z_i)}|Du|^{p}dxdt,
	$$
	then due to $K_{r_i}^\lambda(z_i)\subset K_{s_2R}(z_0)$, we have
	$$
	\begin{aligned}
	|K_{r_i}^\lambda(z_i)|	&\leq \f{2^p}{\lambda^p}\int_{K_{r_i}^\lambda(z_i)\backslash E_{s_2}(\lambda/4)}|Du|^{p}dxdt+\f{2^p}{\lambda^p}\int_{K_{r_i}^\lambda(z_i)\cap E_{s_2}(\lambda/4)}|Du|^{p}dxdt\\
	& \leq \f{|K_{r_i}^\lambda(z_i)|}{4^p}+ \f{1}{\lambda^p}\int_{K_{r_i}^\lambda(z_i)\cap E_{s_2}(\lambda/4)}|Du|^{p}dxdt.
	\end{aligned}
	$$
	This implies 
	\begin{equation}\label{eq-Kwi}
	|K_{r_i}^\lambda(z_i)|\lesi \f{1}{\lambda^p}\int_{K_{r_i}^\lambda(z_i)\cap E_{s_2}(\lambda/4)}|Du|^{p}dxdt.
	\end{equation}
		
	By Holder's inequality, we have
	$$
	\begin{aligned}
	\Big(&\int_{K_{r_i}^\lambda(z_i)\cap E_{s_2}(\lambda/4)}|Du|^{p}dxdt\Big)\\
	&\leq \Big(\fint_{K_{r_i}^\lambda(z_i)\cap E_{s_2}(\lambda/4)}|Du|^{p(1+\epsilon_0)}dxdt\Big)^{\f{1}{1+\epsilon_0}}|K_{r_i}^\lambda(z_i)|
	\Big(\f{|K_{r_i}^\lambda(z_i)\cap E_{s_2}(\lambda/4)|}{|K_{r_i}^\lambda(z_i)|}\Big)^{1-\f{1}{1+\epsilon_0}},
	\end{aligned}
	$$
	where $\epsilon_0$ is a constant in Propositions \ref{higherInte-Prop1-inter} and \ref{higherInte-Prop1-boundary}.

	By Propositions \ref{higherInte-Prop1-inter} and \ref{higherInte-Prop1-boundary}, the scaled mappings \eqref{Scaledmap1}, \eqref{Scaledmap2}, \eqref{Scaledmap3} and Lemma \ref{coveringlemma}, we obtain
	$$
	\begin{aligned}
	\Big(\fint_{K_{r_i}^\lambda(z_i)\cap E_{s_2}(\lambda/4)}|Du|^{p(1+\epsilon_0)}dxdt\Big)^{\f{1}{1+\epsilon_0}}&\leq 	\Big(\fint_{K_{r_i}^\lambda(z_i)}|Du|^{p(1+\epsilon_0)}dxdt\Big)^{\f{1}{1+\epsilon_0}}\\
	&\lesi \lambda^p.
	\end{aligned}
		$$
		Inserting these two estimates into  \eqref{eq-Kwi} we get that
		$$
		|K_{r_i}^\lambda(z_i)|\lesi |K_{r_i}^\lambda(z_i)|
		\Big(\f{|K_{r_i}^\lambda(z_i)\cap E_{s_2}(\lambda/4)|}{|K_{r_i}^\lambda(z_i)|}\Big)^{1-\f{1}{1+\epsilon_0}}.
		$$
		This implies
		$$
		|K_{r_i}^\lambda(z_i)|\lesi |K_{r_i}^\lambda(z_i)\cap E_{s_2}(\lambda/4)|.
		$$
		This completes our proof.
\end{proof}

%We now set
%$$
%\mathcal{I}:=\{i: Q_{20r_i}^\lambda(x_i)\subset \Om_T\}, \ \text{and} \ \  %\mathcal{J}:=\{i: Q_{20r_i}^\lambda(z_i)\cap \Om_T^c\neq \emptyset\}.
%$$
\begin{lem}\label{lem2-mainproof}
For each $\epsilon>0$, there exist $\delta$ and $A_1$  so that the following holds true.  For each $i\in \mathcal{I}$, there exists $v_i$ defined in $K_{6r_i}^\lambda(z_i)$ satisfying
	
	\begin{equation}
	\label{eq1-proof ENlambda}
	\|D\bar{v}_i\|_{L^\vc(K_{6r_i}^\lambda(z_i))}\leq A_1\lambda^p,
	\end{equation}
	and
	\begin{equation}
	\label{eq2-proof ENlambda}
	\fint_{K_{6r_i}^\lambda(z_i)} |D(u-\bar{v}_i)|^pdxdt \leq \epsilon^p\lambda^p.
	\end{equation}

\end{lem}
\begin{proof}
	For $r>0$, $\lambda>1$ and $x\in \Om$ we define 
	$$
	B_r^\lambda(x)=\begin{cases}
	B_{r}(x),  \ \ \ \ &p\geq 2\\
	B_{\lambda^{\f{p-2}{2}}r}(z),  \ \ \ \ & \f{2n}{n+2}<p< 2.
	\end{cases}
	$$
	
	For each $i$, We now consider two cases $B_{10r_i}^\lambda(z_i)\cap \partial\Om\neq \emptyset$ and $B_{10r_i}^\lambda(z_i)\subset \Om$. 
	
	We now consider	the first case $B_{10r_i}^\lambda(z_i)\cap \partial\Om\neq \emptyset$. We now consider the following the standard scaled mappings
	\begin{equation}\label{Scaledmap1}
	\left\{
	\begin{aligned}
	&\bar{u}_i(x,t)= \f{u(x_i+2r_i(x-x_i),t_i+\lambda^{2-p}(2r_i)^2(t-t_i))}{2r_i\lambda},\\
	&\bar{F}_i(x,t)= \f{F(x_i+2r_i(x-x_i),t_i+\lambda^{2-p}(2r_i)^2(t-t_i))}{\lambda},\\
	&\bar{{\bf a}}_i(\xi,x,t)= \f{{\bf a}(\lambda \xi,x_i+2r_i(x-x_i),t_i+\lambda^{2-p}(2r_i)^2(t-t_i))}{\lambda^{p-1}},
	\end{aligned}\right.
	\end{equation}
	\[
	\tilde{\Om}=\Big\{x_i+\f{x-x_i}{2r_i}:x\in \Om\Big\},
	\]
	as $p\geq 2$, and
	\begin{equation}\label{Scaledmap2}
	\left\{
	\begin{aligned}
	&\bar{u}_i(x,t)= \f{u(x_i+2\lambda^{\f{p-2}{2}}r_i(x-x_i),t_i+(2r_i)^2(t-t_i))}{2r_i\lambda^{p/2}},\\
	&\bar{F}_i(x,t)= \f{F(x_i+2\lambda^{\f{p-2}{2}}r_i(x-x_i),t_i+(2r_i)^2(t-t_i))}{\lambda},\\
	&\bar{{\bf a}}_i(\xi,x,t)= \f{{\bf a}(\lambda \xi,x_i+2\lambda^{\f{p-2}{2}}r_i(x-x_i),t_i+(2r_i)^2(t-t_i))}{\lambda^{p-1}},\\
	\end{aligned}\right.
	\end{equation}
	\[
	\tilde{\Om}=\left\{x_i+\f{x-x_i}{2\lambda^{\f{p-2}{2}}r_i}:x\in \Om\right\},
	\]
	as $\f{2n}{n+2}<p<2$, where $z_i=(y_i, t_i)$.
	
	By a simple calculation, it can be verified that $\bar{u}_i$ solves
	$$
	(\bar{u}_i)_t-\di\, \bar{\mathbf{a}}_i(D \bar{u}_i,x,t)=\di\, (|\bar{F}_i|^{p-2}\bar{F}_i) \quad \text{in} \quad K_{5}(z_i).
	$$
	%From Lemma \ref{coveringlemma}, we have
	%\begin{equation}\label{eq1-proof approx mainresult}
	%\fint_{Q_{20r_i}^\lambda(z_i)}\Big[|Du|^{p}+\f{1}{\delta}|F|^{p}\Big] dxdt<\lambda^p.
	%\end{equation}
	
	Sine $B_{5}(x_i)\cap \partial \tilde\Om \neq \emptyset$, there is a point $x_0^i\in B_{5}(x_i)\cap \partial \tilde\Om$. From the definition of the $(\delta, R_0)$ Reifenberg flat domain, it can be seen that there exists a coordinate system, whose variables are still denoted by $x=(x_1,\dots,x_n)$ with origin at some interior point of $\Omega$ such that in this new coordinate system $x^i_0=(0,\ldots,0, -\f{10 \delta}{1-\delta})$ and we have
	\begin{equation}\label{eq4-proof-prop1mainthm}
	B_{15}^+ \subset \Omega\cap B_{15}\subset B_{15}\cap \{x: x_n>-180 \delta \}.
	\end{equation}
	If we choose $\delta<\f{1}{1000}$, then we have 
	\begin{equation}\label{balls relations}
	K_{5}(z_i)\subset K_{15}(0,t_i)\subset K_{20}(z_i).%\subset K_{8 \chi r_i}^\lambda(0,t_i).
	\end{equation}
	
	This along with Lemma \ref{coveringlemma} 
	implies that 
	\begin{equation*}
	\fint_{K_{15}(0,t_i)}|D\bar u|^{p} dz \lesi \lambda, \ \ \text{and} \ \ \fint_{K^\lambda_{15}(0,t_i)}|\bar F|^{p} dz \lesi\delta\lambda.
	\end{equation*}
	
	Therefore, by Corollary \ref{cor1} there exists $\delta$ so that we can find $\bar v_i$ defined in $K_9(0,t_i)$ satisfying
	$$
	\|\bar v_i\|_{L^\vc(K_9(0,t_i))}\lesi 1,
	$$
	and
	$$
	\fint_{K_9(0,t_i)}|D(\bar u_i-\bar v_i)|^pdxdt\leq \epsilon^p.
	$$
	On the other hand, we have $K_3(z_i)\subset K_9(0,t_i)\subset K_{15}(z_i)$, we imply that there exist $\delta$ and $A_1$ so that
	$$
	\|\bar v_i\|_{L^\vc(K_3(z_i))}\leq A_1,
	$$
	and
	$$
	\fint_{K_3(z_i)}|D(\bar u_i-\bar v_i)|^pdxdt\leq \epsilon^p.
	$$
We now consider the following rescaled map:
\begin{equation}\label{Scaledmap3}
	v_i(x,t)=\left\{
	\begin{aligned}
	&2r_i\lambda \bar{v}_i\left(x_i+\f{x-x_i}{2r_i},t_i+\f{t-t_i}{\lambda^{2-p}(2r_i)^2}\right), \ \ \ p\geq 2,\\
	&2r_i\lambda^{p/2} \bar{v}_i\left(x_i+\f{x-x_i}{2\lambda^{\f{p-2}{2}}r_i},t_i+\f{t-t_i}{(2r_i)^2}\right), \ \ \ \f{2n}{n+2}<p< 2.\\
		\end{aligned}\right.
\end{equation}
	Then it is easy to see that $v_i$ is defined in $K_{6r_i}^\lambda(z_i)$ and satisfies \eqref{eq1-proof ENlambda} and \eqref{eq2-proof ENlambda}.
	
	\medskip
	
	The case $B_{10r_i}^\lambda(z_i)\subset \Om$ can be done in the same manner with making use of Corollary \ref{cor1} instead of Proposition \ref{appProp3-inter}. Hence, we omit details.
\end{proof}

\begin{prop}
	\label{prop2}
	There exists $N_0>1$ so that for any $\lambda>\tilde{C}_0\lambda_0$ we have 
	\begin{equation}\label{eq-ENlambda}
	\begin{aligned}
	|E_{s_1}(N_0\lambda)|\leq &\epsilon^p \left[|E_{s_2}(\lambda/4)|  +\f{1}{(\delta\lambda)^\eta}\int_{\delta \lambda/4}^\vc t^{\eta}|\{z\in K_{s_2R}(z_0): |F(z)|>t\}|\f{dt}{t}\right].
	\end{aligned}
	\end{equation}
\end{prop}
\begin{proof}
	
Since $E_{s_1}(N_0\lambda)\subset E(\lambda)$, we have
$$
E_{s_1}(N_0\lambda)=\{z\in E(\lambda): |Du(z)|>N_0\lambda\}.
$$
This along with Lemma \ref{coveringlemma} implies
$$
|E_{s_1}(N_0\lambda)|\leq \sum_{i}|\{z\in K_{6r_i}^\lambda(z_i): |Du(z)|>N_0\lambda\}|.
$$
Taking $N_0=(2A_1)^{1/p}$, from Lemma \ref{lem2-mainproof} we have
$$
\begin{aligned}
\sum_{i}&|\{z\in K_{6r_i}^\lambda(z_i): |Du(z)|>N_0\lambda\}|\\
&\leq \sum_{i}\Big[\Big|\Big\{z\in K_{6r_i}^\lambda(z_i): |D(u-v_i)(z)|>\f{N_0\lambda}{2}\Big\}\Big|+\Big|\Big\{z\in K_{6r_i}^\lambda(z_i): |Dv_i(z)|>\f{N_0\lambda}{2}\Big\}\Big|\Big]\\
&=\sum_{i}\Big|\Big\{z\in K_{6r_i}^\lambda(z_i): |D(u-v_i)(z)|>\f{N_0\lambda}{2}\Big\}\Big|\\
&\leq \sum_{i} \f{2^p}{(N_0\lambda)^p}\int_{K^{\lambda}_{6r_i}(z_i)}|D(u-v_i)|^pdxdt\\
&\leq c\epsilon^p \sum_{i} |K^\lambda_{6r_i}(z_i)|\\
&\leq c\epsilon^p \sum_{i} |K^\lambda_{r_i}(z_i)|.
\end{aligned}
$$
This, Proposition \ref{lem2-mainproof} and the disjointness of the family $\{K_{r_i}^{\lambda}(z_i)\}$ imply the estimate \eqref{eq-ENlambda} as desired.
\end{proof}

We now recall the following auxiliary lemma in \cite[Lemma 4.3]{HL}.
\begin{lem}\label{HF'sLemma}
	Let $f$ be a bounded nonnegative function on $[a_1, a_2]$ with $0<a_1<a_2$. Assume that for any $a_1\leq x_1\leq x_2\leq a_2$ we have
	$$
	f(x_1)\leq \theta_1 f(x_2)+\f{A_1}{(x_2-x_1)^{\theta_2}}+A_2,
	$$
	where $A_1, A_2>0$, $0<\theta_1<1$ and $\theta_2>0$. Then, there exists $c=c(\theta_1,\theta_2)$ so that
	$$
	f(x_1)\leq c\Big[\f{A_1}{(x_2-x_1)^{\theta_2}}+A_2\Big].
	$$
\end{lem}

We now ready to give the proof of Theorem \ref{mainthm1}.

\begin{proof}[Proof of Theorem \ref{mainthm1}:]
	Since $\Om_T$ is a bounded domain. It suffices to prove that there exists a constant $C>0$ independing of $u$ and $x_0$ so that
	$$
	\| |Du|\|_{L^{q,r}(K_{R}(z_0))}\leq C\left[  \|F\|_{L^{q,r}(\Om_T)}+\|F\|_{L^{q,r}(\Om_T)}^{d}\right].
	$$
	
	For each $k>0$ we define $|Du|_k=\min\{k, |Du|\}$. Then $|Du|_k\in L^{q,r}(\Om_T)$ for all $q>p, 0<r\leq \vc$. We set $E^k_s(\lambda)=\{z\in K_{sR}(z_0): |Du(z)|_k>\lambda\}$ for $s>0$.
	
	From \eqref{eq-ENlambda}, it follows immediately that there exists $C$ independing of $k$ so that
	\begin{equation}\label{eq-Duk}
	|E^k_{s_1}(N_0\lambda)|\leq C\epsilon^p\Big[ |E^k_{s_2}(\lambda/4)|  +\f{1}{(\delta\lambda)^\eta}\int_{\delta \lambda/4}^\vc t^{\eta}|\{z\in K_{s_2R}(z_0): |F(z)|>t\}|\f{dt}{t}\Big]
	\end{equation}
	
	\noindent{\bf Case 1: $0<r<\vc$.} We have
	$$
	\begin{aligned}
	\| |Du|_k\|_{L^{q,r}(K_{s_1R}(z_0))}^r&=C\int_0^\vc \left[\lambda^q|\{z\in K_{s_1R}(z_0): |Du(z)|_k>N_0\lambda\}|\right]^{r/q}\f{d\lambda}{\lambda}\\
	&=\int_0^{\tilde{C}_0\lambda_0}\ldots +\int_{\tilde{C}_0\lambda_0}^\vc\ldots:=I_1+I_2. 
	\end{aligned}
	$$
	It is easy to see that 
	$$
	\begin{aligned}
    I_1&\leq C|\Om_T|^{r/q} (\tilde{C}_0\lambda_0)^r= C(s_2-s_1)^{-\f{(n+2)dr}{p}}\Big[\Big(\fint_{K_{2R}(z_0)}|Du|^{p}dxdt\Big)^{\f{dr}{p}}+\Big(\f{1}{\delta^{\eta}}\fint_{K_{2R}(z_0)}(|F|^{\eta}+1)dxdt\Big)^{\f{dr}{\eta}}\Big]\\
    &\leq C(s_2-s_1)^{-\f{(n+2)dr}{p}}\Big(\int_{\Om_T} |F|^\eta dxdt\Big)^{\f{rd}{\eta}}\\
    &\leq C(s_2-s_1)^{-\f{(n+2)dr}{p}}\|F\|_{L^{q,r}(\Om_T)}^{rd}.
	\end{aligned}
	$$
	To take care of the second term $I_2$, we apply \eqref{prop2} to write
	$$
	\begin{aligned}
	I_2\leq &C\epsilon^{p r/q}\int_{C_0}^\vc [\lambda^{q}|\{z\in K_{s_2R}(z_0): |Du(z)|_k >\lambda/4\}|]^{r/q}\f{d\lambda}{\lambda}\\
	& \ \ \ \ + C\epsilon^{p r/q}\int_{\tilde{C}_0\lambda_0}^\vc  \lambda^{(q-\eta)r/q}\left[\int_{\delta\lambda/4}^\vc t^{\eta}|\{z\in \Om_T: |F(z)|>t\}|\f{dt}{t}\right]^{r/q}\f{d\lambda}{\lambda}\\
	&= I_{21}+I_{22}.
	\end{aligned}
	$$
	Obviously, $I_{21}\leq C_1\epsilon^{p r/q} \| |Du|_k\|_{L^{q,r}(K_{s_1R}(z_0))}^r$. In order to take case of the second term, we consider three cases.
	
    \textbf{Subcase 1.1: $q<r<\vc$.}
    Applying Hardy's inequality (see for example \cite[Theorem 330]{HLP}), we obtain
    $$
    \begin{aligned}
    I_{22}&\leq C(\delta)\epsilon^{p r/q}\int_0^\vc \lambda^{(q-\eta)r/q}\lambda^{\eta r/q}|\{z\in \Om_T: |F(z)|>\lambda\}|^{r/q}\f{d\lambda}{\lambda}\\
    &=C(\delta)\epsilon^{p r/q}\int_0^\vc \lambda^{r}|\{z\in \Om_T: |F(z)|>\lambda\}|^{r/q}\f{d\lambda}{\lambda}\\
    &=C\|F\|^r_{L^{q,r}(\Om_T)}.
    \end{aligned}
    $$ 
   Hence,
   $$
   \begin{aligned}
   \||Du|_k\|^r_{L^{q,r}(K_{s_1R}(z_0))}\leq& C_1\epsilon^{p r/q}\||Du|_k\|^r_{L^{q,r}((K_{s_2R}(z_0))}\\
   & \ \ +C(\|F\|_{L^{q,r}(\Om_T)}^r+(s_2-s_1)^{-\f{(n+2)dr}{p}}\|F\|_{L^{q,r}(\Om_T)}^{rd}).
   	\end{aligned}
   	$$
   By choosing $\epsilon$ so that $C_1\epsilon^{p r/q}<1$ and then and applying 
   Lemma \ref{HF'sLemma} for  $f(s)=\||Du|_k\|_{L^{q,r}(K_{sR}(z_0))}$, $a_1=1, a_2=2$, $A_1=\|F\|_{L^{q,r}(\Om_T)}^{d}$ and $A_2=\|F\|_{L^{q,r}(\Om_T)}$ we deduce that there exists $C$ independent of $k$ so that 
   $$
   \| |Du|_k \|_{L^{q,r}(K_{s_1R}(z_0))}\lesi  \|F\|_{L^{q,r}(\Om_T)}+(s_2-s_1)^{-\f{(n+2)d}{p}}\|F\|_{L^{q,r}(\Om_T)}^{d}.
   $$
   This implies 
   $$
   \| |Du|_k \|_{L^{q,r}(K_{R}(z_0))}\lesi  \|F\|_{L^{q,r}(\Om_T)}+\|F\|_{L^{q,r}(\Om_T)}^{d}.
   $$
   Letting $k\to \vc$, we obtain
   $$
   \| |Du|\|_{L^{q,r}(K_{R}(z_0))}\lesi  \|F\|_{L^{q,r}(\Om_T)}+\|F\|_{L^{q,r}(\Om_T)}^{d}.
   $$
   
   \textbf{Subcase 1.2: $0<r\leq q$.} In order to deal with this case, we need the following variant of reverse-H\"older's inequality in \cite[Lemma 3.5]{Ba}
   
   \begin{lem}
   	\label{lem-reverseHolder}
   	Let $h:[0,\vc)\to [0,\vc)$ be a non-decreasing, measurable functions and let $1\leq \alpha\leq \vc$ and $r>0$. Then there exists $C>0$ so that for any $\lambda>0$ we have
   	$$
   	\Big[\int_{\lambda}^{\vc}\left(t^r h(t)\right)^\alpha \f{dt}{t}\Big]^{1/\alpha}\leq \lambda^r h(\lambda)+C\int_{\lambda}^{\vc}t^r h(t)\f{dt}{t}, \ \ \ \alpha<\vc,
   	$$
   	and
   	$$
   	\sup_{t>\lambda }t^r h(t)\leq C\lambda^r h(\lambda)+C\int_{\lambda}^{\vc}t^r h(t)\f{dt}{t}, \ \ \ \alpha=\vc.
   	$$
   \end{lem}
   
   We now apply Lemma \ref{lem-reverseHolder} to obtain
   $$
   \begin{aligned}
   \left[ \int_{\delta\lambda/4}^\vc t^{\eta}|\{z\in K_{s_2R}(z_0): |F(z)|>t\}|\f{dt}{t}\right]^{r/q}&=\left\{\int_{\delta\lambda/4}^\vc \left[t^{\eta r/q}|\{z\in K_{s_2R}(z_0): |F(z)|>t\}|^{r/q}\right]^{q/r}\f{dt}{t}\right\}^{r/q}\\
   &\leq C\lambda^{\eta r/q}|\{z\in \Om_T: |F(z)|>\delta\lambda/4\}|^{r/q}\\
   &\ \ \ +C\int_{\delta\lambda/4}^\vc t^{\eta r/q}|\{z\in \Om_T: |F(z)|>t\}|^{r/q}\f{dt}{t}.
   \end{aligned}
   $$
   Inserting this into the expression of $I_{21}$, we get that
   $$
   \begin{aligned}
   I_{22}&\leq   C\epsilon^{p r/q}\int_{0}^\vc \lambda^{(q-\eta)r/q}\lambda^{\eta r/q}|\{z\in \Om_T: |F(z)|>\delta\lambda/4\}|^{r/q}\f{d\lambda}{\lambda}\\
   &\ \ \ +C\epsilon^{p r/q}\int_{0}^\vc \lambda^{(q-\eta)r/q}\int_{\delta\lambda/4}^\vc t^{\eta r/q}|\{z\in K_{s_2R}(z_0): |F(z)|>t\}|^{r/q}\f{dt}{t}\f{d\lambda}{\lambda}:=I_{22}^1+I_{22}^2.
   \end{aligned}
   $$
   Obviously,
   $$
   I_{22}^2\leq C \epsilon^{p r/q}\int_{0}^\vc \lambda^{(q-\eta)r/q}\lambda^{\eta r/q}|\{z\in \Om_T: |F(z)|>\delta\lambda\}|^{r/q}\f{d\lambda}{\lambda}\sim  \|F\|^r_{L^{q,r}(\Om_T)}.
   $$
   Using Fubini's theorem, we can dominate the term $I_{21}^2$ by
   $$
   \begin{aligned}
   I_{22}^2&\leq C\epsilon^{p r/q}\int_{0}^\vc \lambda^{(q-\eta)r/q}\int_{\lambda/4}^\vc t^{\eta r/q}|\{z\in \Om_T: |F(z)|>t\}|^{r/q}\f{dt}{t}\f{d\lambda}{\lambda}\\
   &=C\epsilon^{p r/q}\int_{0}^\vc t^{\eta r/q}|\{z\in \Om_T: |F(z)|>t\}|^{r/q}\int_0^{\f{t}{4\delta}}\lambda^{(q-\eta)r/q} \f{d\lambda}{\lambda}\f{dt}{t}\\
   &=C\int_{0}^\vc t^{r}|\{z\in \Om_T: |F(z)|>t\}|^{r/q}\f{dt}{t}\\
   &=   C\|F\|^r_{L^{q,r}(\Om_T)}.
	\end{aligned}
	$$
	Taking these teo estimates $I_{22}^1$ and $I_{22}^2$ into account, we imply that
		$$
	I_{22}\leq C\|F\|^r_{L^{q,r}(\Om_T)}.
	$$
	Therefore,
	$$
	\||Du|_k\|^r_{L^{q,r}(K_{s_1R}(z_0))}\leq C_2\epsilon^{p r/q}\||Du|_k\|^r_{L^{q,r}(K_{s_2R}(z_0))}+C(\|F\|_{L^{q,r}(\Om_T)}^r+(s_2-s_1)^{-\f{(n+2)dr}{p}}\|F\|_{L^{q,r}(\Om_T)}^{rd}).
	$$
	By choosing $\epsilon$ so that $C_2\epsilon^{p r/q}<1$, and arguing similarly to the Subcase 1.1, we obtain
	$$
	\| |Du|\|_{L^{q,r}(K_{R}(z_0))}\lesi  \|F\|_{L^{q,r}(\Om_T)}+\|F\|_{L^{q,r}(\Om_T)}^{d}.
	$$
		\medskip
	
	\noindent\textbf{Case 2: $r=\vc$.} In this situation, we have
	$$
	\begin{aligned}
	\||Du|_k\|_{L^{q,\vc}(K_{s_1R(z_0)})}&\le C\sup_{\lambda>0}\lambda|\{z\in \Omega_T: |Du(z)|_k>N_0\lambda\}|^{1/q}\\
	&=\sup_{0<\lambda\leq \tilde{C}_0\lambda_0}\ldots + \sup_{\lambda>\tilde{C}_0\lambda_0}\ldots=J_1+J_2. 
	\end{aligned}
	$$
	Obviously, we have
	$$
	\begin{aligned}
	J_1&\leq C|\Om_T|^{1/q} \tilde{C}_0\lambda_0= C(s_2-s_1)^{-\f{(n+2)d}{qp}}\Big(\int_{\Om_T} |F|^p +|Du|^pdxdt\Big)^{\f{d}{p}}\\
	&\lesi C(s_2-s_1)^{-\f{(n+2)d}{qp}}\Big(\int_{\Om_T} |F|^pdxdt\Big)^{\f{d}{p}}\\
	&\lesi C(s_2-s_1)^{-\f{(n+2)d}{qp}}\|F\|_{L^{q,\vc}(\Om_T)}^{d}.
	\end{aligned}
	$$
	For the second term, applying \eqref{eq-Duk} we obtain
	$$ 
	\begin{aligned}
	J_2&\leq \sup_{\lambda>\tilde{C}_0\lambda_0}\lambda\Big\{c\epsilon^p|\{z\in K_{s_2R}(z_0): |Du(z)|_k>t\}|+\f{c\epsilon^p}{(\delta\lambda)^\eta}\int_{\delta \lambda/4}^\vc t^{\eta}|\{z\in \Om_T: |F(z)|>t\}|\f{dt}{t}\Big]\Big\}^{1/q}\\
	&\leq c\epsilon^{p/q}\||Du|_k\|_{L^{q,\vc}(K_{s_2R}(z_0))} +\f{c\epsilon^{p/q}}{\delta^{\eta/q}}\Big[\lambda^{q-\eta}\int_{\delta \lambda/4}^\vc t^{\eta-q}t^{q}|\{z\in \Om_T: |F(z)|>t\}|\f{dt}{t}\Big]^{1/q}\\
	&\leq  c\epsilon^{p/q}\||Du|_k\|_{L^{q,\vc}(K_{s_2R}(z_0))}+\f{c\epsilon^{p/q}}{\delta^{\eta/q}}\Big[\lambda^{q-\eta}\int_{\delta \lambda/4}^\vc t^{\eta-q}\|F\|_{L^{q,\vc}(\Om_T)}^q\f{dt}{t}\Big]^{1/q}\\
	&\leq C_3\epsilon^{p/q}\||Du|_k\|_{L^{q,\vc}(K_{s_2R}(z_0))}+C\|F\|_{L^{q,\vc}(\Om_T)}.
	\end{aligned}
	$$
	Taking $\epsilon$ so that $C_3\epsilon^{p/q}<1$ and and arguing similarly to the Subcase 1.1, we obtain
	$$
	\| |Du|\|_{L^{q,\vc}(K_{R}(z_0))}\lesi  \|F\|_{L^{q,\vc}(\Om_T)}+\|F\|_{L^{q,\vc}(\Om_T)}^{d}.
	$$

	This completes our proof.
	\end{proof}
	
\bigskip

\textbf{Acknowledgement.} The authors would like to thank the referee for
useful comments and suggestions to improve the paper.
The first named author was supported by the research grant ARC DP140100649 from the Australian Research Council and
Vietnam's National Foundation for Science and Technology Development
(NAFOSTED) under Project 101.02--2016.25. The second named author was supported by the research grant ARC DP140100649.

\end{document}